\documentclass[12pt]{article}

  \usepackage{amsmath}
\usepackage{amsfonts}
\usepackage{latexsym}
\usepackage{amssymb}
\usepackage{amsthm}
\usepackage[normalem]{ulem}
\usepackage{csquotes}
\usepackage{lineno,xcolor}
\usepackage[normalem]{ulem}
\usepackage{filecontents}
\usepackage{scalefnt}
\usepackage{tikz}
\usepackage{fancyhdr}
\usetikzlibrary{arrows}
\usetikzlibrary{positioning}
\usepackage[symbol]{footmisc}

\usepackage{lineno,xcolor}
\definecolor{amaranth}{rgb}{0.9, 0.17, 0.31}
\definecolor{bluegray}{rgb}{0.4, 0.6, 0.8}

\makeatletter

\newtheorem*{maintheorem*}{Main Theorem}
\newtheorem{theorem}{Theorem}[section]
\newtheorem{proposition}[theorem]{Proposition}

\newtheorem{lemma}[theorem]{Lemma}

\newtheorem*{theorem*}{Theorem}

\newtheorem{remark}[theorem]{Remark}

\newtheorem*{example*}{Example}
\newtheorem*{conjecture*}{Conjecture}
\def\1{\mathbf 1}

\def\j{\mathbf j}

\def\0{\mathbf 0}

\def\cB{\mathcal B}
\def\cC{\mathcal C}

\def\cL{\mathcal L}

\def\cP{\mathcal P}
\def\cX{\mathcal X}

\def\cS{\mathcal S}

\def\<{\langle}
\def\>{\rangle}
\newcommand\comment[1]{}

\newcommand*{\shifttext}[2]{
  \settowidth{\@tempdima}{#2}
  \makebox[\@tempdima]{\hspace*{#1}#2}
}

\newcommand\redsout{\bgroup\markoverwith{\textcolor{amaranth}{\rule[0.5ex]{2pt}{0.4pt}}}\ULon}

 \makeatletter
\def\@fnsymbol#1{\ensuremath{\ifcase#1\or *\or \dagger\or \ddagger\or
   \mathsection\or \mathparagraph\or \|\or **\or \dagger\dagger
   \or \ddagger\ddagger \else\@ctrerr\fi}}
    \makeatother

\title{Reconstructing a generalized quadrangle with a hemisystem from a $4-$class association scheme}
\author{Giusy Monzillo\thanks{The research was supported by the Italian National Group for Algebraic and Geometric Structures and their Applications (GNSAGA-INdAM).} 
\\
\small {\tt giusy.monzillo@unibas.it}\\
\small Dipartimento di Matematica, Informatica ed Economia\\[-0.8ex]
\small Universit\`a degli Studi della Basilicata\\[-0.8ex]
\small Viale dell'Ateneo Lucano 10 - 85100 Potenza (Italy)
\small Potenza, Italy\\
}
\date{}

\begin{document}


\maketitle

\thispagestyle{fancy}
\fancyhf{}
\renewcommand{\headrulewidth}{0pt}
\lhead{}

\begin{abstract}

In 2013, van Dam, Martin and Muzychuk  constructed a cometric $Q-$ antipodal $4-$class association scheme from a GQ of order $(t^2,t)$, $t$ odd,  which have a hemisystem. In this paper we characterize this scheme by its Krein array. The techniques which are used involve the triple intersection numbers introduced by Coolsaet and Juri\v{s}i\'c.

\end{abstract}


{\it Math. Subj. Class.: 05E30, 51E20}

\section{Introduction}

A (finite) {\em partial geometry}  is an incidence structure $\cS=(\cP,\cB, {\rm I})$, where  $\cP$ and $\cB$ are disjoint (nonempty) sets of objects called {\em points} and {\em lines}, respectively, and for which $\rm I\subseteq (\cP,\cB)\cup(\cB,\cP)$ is a symmetric point-line {\em incidence relation} satisfying the following axioms:
\begin{itemize} 
\item[ (i)] each point is incident with $t+1$ lines, and two distinct points are incident with at most one line,
\item[(ii)]  each line is incident with $s+1$ points, and two distinct lines are incident with at most one point,
\item[(iii)] if $x$ is a point and $L$ is a line not incident with $x$, then there are exactly $\alpha$ ($\alpha\ge 1$) pairs $(y, M)\in \cP\times \cB$ such that $x\, {\rm I}\, M\, {\rm I}\, y\, {\rm I}\,  L$.
\end{itemize}

The integers $s$, $t$ and $\alpha$ are the {\em parameters} of the partial geometry. When $\alpha=1$, the partial geometry $\cS$ is a {\em generalize quadrangle} (GQ), and it is said to have {\em order} $(s,t)$.  If $\cS$ has order $(s,t)$, then  $|\cP| = (s + 1)(st + 1)$ and $|\cB| = (t + 1)(st + 1)$.

Partial geometries were introduced by Bose \cite{bo} to generalize known characterization theorems for strongly regular graphs. Indeed, the point graph of a partial geometry, that is the graph with vertex set $\cP$ whose edges are the pairs of collinear points,  is a strongly regular graph, i.e., a $2-$class association scheme. Such strongly regular graphs are called {\em geometric}. A strongly regular graph is called {\em pseudo-geometric}  if it has the same parameters  as a geometric strongly regular graph. The fundamental problem posed by Bose is to establish when a pseudo-geometric graph is geometric. Contributions to this problem can be found, for example, in \cite{cgs,h80}.

On the other hand, when a generalized quadrangle satisfies certain properties, it is possible to  construct association schemes with more than two classes. Therefore, it is natural to ask if such association schemes are characterized by their parameters.

Payne \cite{payne} constructed a $3-$class association scheme starting from a generalized quadrangle with a quasi-regular point. Subsequently, Hobart and Payne \cite{HobPay} proved that an association scheme having the same parameters and satisfying an assumption about maximal cliques must be the above $3-$class scheme.
Ghinelli and L\"owe \cite{gl} defined a $4-$class association scheme on the points of a generalized quadrangle with a regular point, and they characterize the scheme by its parameters. Penttila and Williford \cite{pw} constructed an infinite family of $4-$class association schemes starting from  a generalized quadrangle with a doubly subtended subquadrangle. A characterization of these schemes can be found in \cite{ms}. 

van Dam, Martin and Muzychuk \cite{vdmm} constructed a cometric $Q-$antipodal $4-$class association scheme from a GQ of order $(t^2,t)$, $t$ odd,  which has a hemisystem, defined in Section \ref{sec_2}. The  Krein array is also given by the authors, and used in this paper in order to characterize the scheme. The techniques which are used involve the triple intersection numbers introduced by Coolsaet and Juri\v{s}i\'c \cite{cj}, and they basically retrace the approach already successfully adopted in \cite{ms}.

 The principal references on association schemes are \cite{bi,bcn}. 
 
 \bigskip 
The paper is structured as follows. Section \ref{sec_2} contains background information on hemisystems in generalized quadrangles as well as the parameters of the van Dam-Martin-Muzychuk $4-$class scheme arising from this geometry. Let $\cX=(X, \{R_i\}_{i=0,...,4})$ be a scheme with the same parameters. The data suggest that $R_1\cup R_2$ can be viewed as the collinearity between points in the dual of the GQ to be reconstructed. In Section \ref{sec_3} properties of maximal $\{0,1,2\}$-cliques of the scheme are explored. In particular, by  considering {\em triple intersection numbers}, we prove that for any pair in  $R_1\cup R_2$  there exists a unique maximal $\{0,1,2\}$-clique (of size $t+1$) containing it (Proposition \ref{lem_3}). In Section \ref{sec_4},  we reconstruct a GQ of order $(t,t^2)$, $t$ odd, having a set of points such that every line meets it in $(t+1)/2$ points. The dual of this GQ is the desired quadrangle (Theorem \ref{main}).

\section{Preliminaries}\label{sec_2}

An association scheme $\cX=(X,\{R_i\}_{i=0}^{d})$, with Krein parameters $q_{ij}^k$, is said to be \emph{cometric}, or \emph{$Q-$polynomial}, if there is an ordering of the primitive idempotents $E_0,\ldots,E_d$ such that $q_{ij}^k\neq0$ implies $k\leq i+j$, and $q_{ij}^{i+j}\neq 0$ for all $i,j,k$. This is equivalent to say that the matrix $L_1^*=(q_{1j}^k)_{k,j=0,\ldots,d}$ is tridiagonal. In this case, one can define the so-called \emph{Krein array} of $\cX$, i.e., the vector \[\{b_0^*,b_1^*,\ldots,b_{d-1}^*;c_1^*,c_2^*,\ldots,c_{d}^*\},\] where $b_i^*=q^i_{1i+1}$ and $c_i^*=q^i_{1i-1}$.
A cometric scheme is \emph{antipodal} if $b_i^*=c_{d-i}^*$ for all $i\neq \big \lfloor \frac{d}{2}\big \rfloor$.

Let $\cS$ be a GQ of order $(t^2,t)$, $t$ odd.  A {\em hemisystem} in $\cS$ is a set of lines  $U$ with the property that every point lies on exactly  $(t+1)/2$ lines of  $U$.  Since the complementary set $U'=\cB\setminus U$ of a hemisystem is also a hemisystem, $$|U|=|U'|=(t^3+1)(t+1)/2.$$ 

Let $\cL=U\cup U'$ and $C$ denote the set of all ordered pairs of distinct intersecting lines from $\cL$. Then, the following relations can be defined on $\cL$:
\begin{itemize}
\item[] $R_0=\{(\ell,\ell):\ell \in \cL\}$,
\item[] $R_1=C\cap ((U\times U)'\cup (U'\times U))$,
\item[] $R_2 =C\cap ((U\times U)\cup (U'\times U'))$,
\item[] $R_3 =((U \times U')\cup (U' \times U))-R_1$,
\item[] $R_4 =((U \times U)\cup (U' \times U'))-R_0-R_2$.
\end{itemize} 

According to Corollary 7.8 in \cite{vdmm}, $\cL$ together with the relations $R_0,R_1,\ldots,R_4$ gives rise to a cometric $Q-$antipodal $4-$class association scheme, whose Krein array is
\begin{equation}\label{Krein}
\{(s+t)(t-1),s^2/t,s(t-1)/t,1;1,s(t-1)/t,s^2/t,(s+t)(t-1)\},
\end{equation}
where $s=t^2-t+1$.

As usual, $n_i$ denotes the valency of the relation $R_i$, $p^k_{ij}$ are the intersection numbers of the scheme, and $P$ and $Q$ are the first and the second eigenmatrices of the scheme, respectively. 

Before calculations, some general considerations are needed.
By using ``dual" arguments with respect to those of the proof of Proposition 2.2.2 in \cite{bcn}, for a $d-$class association scheme, it is possible to compute the matrices $P$ and $Q$, and, in particular, the multiplicities, from one of the matrices $L_i^*=(q_{ij}^k)_{k,j=0,\ldots,d}$, $i=1,\dots,d$, having $d+1$ distinct eigenvalues. If the scheme is cometric, $L_1^*$, whose entries are completely determined by the Krein array,
has the desired property. Once the eigenmatrices are obtained, one can compute all the remaining parameters of the scheme by using Theorem 3.6 in \cite{bi}. 
In conclusion, when a scheme is cometric, it is possible to get all the parameters starting from the Krein array. 

By applying these arguments to our scheme $\cX=(\cL,\{R_i\}_{i=0}^{4})$, whose Krein array is (\ref{Krein}), with $s=t^2-t+1$ and $b=\binom{t+1}{2}$, we have

\vspace{.3in}

\begin{equation*}
 n_1=\frac{(t+1)(t^2+1)}2, \ \ \ \ \ n_2=\frac{(t-1)(t^2+1)}2,  \ \ \ \ \ n_3=\frac{t^2(t^2-1)}2, \ \ \ \ \ n_4=\frac{t^2(t^2+1)}2.
\end{equation*}

\vspace{.3in}

\hspace{-0.2in}
{\begin{minipage}{0.3in}\small
\[
P=\begin{pmatrix}
1 & n_1	 & n_2 & n_3 & n_4 \\[.05in]
1 & b	 & -\frac{(s+1)}2 & -b & \frac{(s-1)}2  \\[.05in]
1 & 0	 & t-1 & 0 & -t  \\[.05in]
1 & -b & -\frac{(s+1)}2 & b & \frac{(s-1)}2   \\[.05in]
1 & -n_1	 & n_2 & -n_3 & n_4
 \end{pmatrix},
 \]
\end{minipage}
}
\hspace{2.5in}
{\begin{minipage}{.5in}\small
\[
Q=\begin{pmatrix}
1 & 2n_2 & s(s+t) & 2n_2 & 1  \\[.05in]
1 & s-1	 & 0 & 1-s & -1 \\[.05in]
1 & -s-1 & 2s & -s-1 & 1 \\[.05in]
1 & -\frac{(s+t)}t	 & 0 & \frac{(s+t)}t & -1  \\[.05in]
1 & \frac{(s-t)}t	 & -\frac{2s}t & \frac{(s-t)}t & -1  
 \end{pmatrix}.
\]
\end{minipage}
}

\newpage 

\hspace{-0.35in}
\bigskip
\resizebox{.95\textwidth}{!}
{\begin{minipage}{\textwidth}
\begin{tabular}{l|ccccc}
$p^1_{ij}$ & \hspace{.01in}  1 &\hspace{.01in} 2 &\hspace{.01in} 3 &\hspace{.01in} 4\\[.1in]
\hline\\[.03in] 
1 & \hspace{-.1in} $0$ &\hspace{.01in} $\frac{t-1}2$ &\hspace{.01in}  $0$ &\hspace{.01in}  $tb$ \\[.1in]
 2 & \hspace{-.1in}  $\frac{t-1}2$ &\hspace{.01in} $0$ &\hspace{.01in}  $\frac{t(s-1)}2$ &\hspace{.01in} $0$  \\[.1in]
 3 & \hspace{-.1in}  $0$ &\hspace{.01in} $\frac{t(s-1)}2$ &\hspace{.01in}  $0$ &\hspace{.01in} $\frac{t^2(s-1)}2$ \\[.1in]
 4 & \hspace{-.1in}  $tb$ &\hspace{.01in} $0$ &\hspace{.01in}  $\frac{t^2(s-1)}2$ &\hspace{.01in} $0$
\end{tabular}
\end{minipage}
}
\hspace{-2.5in}
\resizebox{.95\textwidth}{!}{\begin{minipage}{\textwidth}
\begin{tabular}{l|ccccc}
$p^2_{ij}$ & \hspace{.01in} 1 &\hspace{.01in} 2 &\hspace{.01in} 3 &\hspace{.01in} 4\\[.1in]
\hline\\[.03in] 
1 & \hspace{-.1in} $\frac{t+1}2$  &\hspace{.01in} $0$ &\hspace{.01in}  $tb$ &\hspace{.01in}  $0$ \\[.1in]
 2 & \hspace{-.1in}  $0$ &\hspace{.01in}  $\frac{t-3}2$ &\hspace{.01in} $0$ &\hspace{.01in}  $\frac{t(s-1)}2$   \\[.1in]
 3 & \hspace{-.1in}   $tb$ &\hspace{.01in}  $0$ &\hspace{.01in} $\frac{t^2(s+1)}2$ &\hspace{.01in}$0$ \\[.1in]
 4 & \hspace{-.1in}  $0$ &\hspace{.01in}  $\frac{t(s-1)}2$ &\hspace{.01in} $0$ &\hspace{.01in}  $\frac{t^2(s-3)}2$  
\end{tabular}
\end{minipage}
}

\vspace{.5in}

\hspace{-0.35in}
\resizebox{.8\textwidth}{!}{\begin{minipage}{\textwidth}
\begin{tabular}{l|ccccc}
$p^3_{ij}$ & \hspace{.01in}  1 &\hspace{.01in} 2 &\hspace{.01in} 3 &\hspace{.01in} 4\\[.1in]
\hline\\[.03in] 
1 & \hspace{-.1in} $0$ &\hspace{0.1in} $\frac{t^2+1}2$ &\hspace{0.1in}  $0$ &\hspace{0.1in}  $\frac{t(t^2+1)}2$ \\[.1in]
 2 & \hspace{-.1in}  $\frac{t^2+1}2$ &\hspace{0.1in} $0$ &\hspace{0.1in}  $\frac{(t-2)(t^2+1)}2$ &\hspace{0.1in} $0$  \\[.1in]
 3 & \hspace{-.1in}  $0$ &\hspace{0.1in} $\frac{(t-2)(t^2+1)}2$ &\hspace{0.1in}  $0$ &\hspace{0.1in} $\frac{(s-1)(t^2+1)}2$ \\[.1in]
 4 & \hspace{-.1in}  $\frac{t(t^2+1)}2$ &\hspace{0.1in} $0$ &\hspace{0.1in}  $\frac{(s-1)(t^2+1)}2$ &\hspace{0.1in} $0$
\end{tabular}
\end{minipage}
}
\hspace{-1.5in} 
\resizebox{.8\textwidth}{!}{\begin{minipage}{\textwidth}
\begin{tabular}{l|ccccc}
$p^4_{ij}$ & \hspace{.01in}  1 &\hspace{.01in} 2 &\hspace{.01in} 3 &\hspace{.01in} 4\\[.1in]
\hline\\[.03in] 
1 & \hspace{-.1in} $\frac{(t+1)^2}2$  &\hspace{.01in} $0$ &\hspace{.01in}  $b(t-1)$ &\hspace{.01in}  $0$ \\[.1in]
 2 & \hspace{-.1in}  $0$ &\hspace{.01in}  $\frac{(t-1)^2}2$ &\hspace{.01in} $0$ &\hspace{.01in}  $\frac{(t-1)(s+1)}2$   \\[.1in]
 3 & \hspace{-.1in}   $b(t-1)$ &\hspace{.01in}  $0$ &\hspace{.01in} ${b(t-1)^2}$ &\hspace{.01in}$0$ \\[.1in]
 4 & \hspace{-.1in}  $0$ &\hspace{.01in}  $\frac{(t-1)(s+1)}2$ &\hspace{.01in} $0$ &\hspace{.01in}  $\frac{(s-1)(t^2+3)}2$  
\end{tabular}
\end{minipage}
}

\vspace{.5in}

\noindent
For the sake of completeness, the Krein paramaters are also reported here, although they will not directly contribute to the characterization of the scheme:

\vspace{.5in}

\hspace{-0.5in}
\resizebox{.75\textwidth}{!}{\begin{minipage}{\textwidth}
\begin{tabular}{l|ccccc}
$t q^1_{ij}$ & \hspace{.01in}  1 &\hspace{.01in} 2 &\hspace{.01in} 3 &\hspace{.01in} 4\\[.1in]
\hline\\[.03in] 
1 & \hspace{-.1in} $(t-1)s-2t$ &\hspace{0.1in} $s^2$ &\hspace{0.1in}  $0$ &\hspace{0.1in}  $0$ \\[.1in]
2 & \hspace{-.1in} $s^2$ &\hspace{0.1in} $(t-1)(s+1)s$ &\hspace{0.1in}  $s^2$ &\hspace{0.1in} $0$  \\[.1in]
3 & \hspace{-.1in} $0$ &\hspace{0.1in} $s^2$ &\hspace{0.1in}  $(t-1)s-2t$ &\hspace{0.1in} $t$ \\[.1in]
4 & \hspace{-.1in} $0$ &\hspace{0.1in} $0$ &\hspace{0.1in}  $t$ &\hspace{0.1in} $0$
\end{tabular}
\end{minipage}
}
\hspace{-1in} 
\resizebox{.75\textwidth}{!}{\begin{minipage}{\textwidth}
\begin{tabular}{l|ccccc}
$t q^2_{ij}$ & \hspace{.01in}  1 &\hspace{.01in} 2 &\hspace{.01in} 3 &\hspace{.01in} 4\\[.1in]
\hline\\[.03in] 
1 & \hspace{-.1in} $(t-1)s$  &\hspace{.01in} $(t-1)^2(s+1)$ &\hspace{.01in}  $(t-1)s$ &\hspace{.01in}  $0$ \\[.1in]
2 & \hspace{-.1in} $(t-1)^2(s+1)$ &\hspace{.01in}  $(t-1)(s+3)s$ &\hspace{.01in} $(t-1)^2(s+1)$ &\hspace{.01in}  $t$   \\[.1in]
3 & \hspace{-.1in} $(t-1)s$  &\hspace{.01in} $(t-1)^2(s+1)$ &\hspace{.01in}  $(t-1)s$ &\hspace{.01in}  $0$ \\[.1in]
4 & \hspace{-.1in} $0$ &\hspace{.01in}  $t$ &\hspace{.01in} $0$ &\hspace{.01in}  $0$  
\end{tabular}
\end{minipage}
}

\vspace{.5in}

\hspace{-0.25in}
\resizebox{.75\textwidth}{!}{\begin{minipage}{\textwidth}
\begin{tabular}{l|ccccc}
$t q^3_{ij}$ & \hspace{.01in}  1 &\hspace{.01in} 2 &\hspace{.01in} 3 &\hspace{.01in} 4\\[.1in]
\hline\\[.03in] 
1 & \hspace{-.1in} $0$ &\hspace{0.1in} $s^2$ &\hspace{0.1in}  $(t-1)s-2t$ &\hspace{0.1in}  $t$ \\[.1in]
2 & \hspace{-.1in} $s^2$ &\hspace{0.1in} $(t-1)(s+1)s$ &\hspace{0.1in}  $s^2$ &\hspace{0.1in} $0$  \\[.1in]
3 & \hspace{-.1in} $(t-1)s-2t$ &\hspace{0.1in} $s^2$ &\hspace{0.1in}  $0$ &\hspace{0.1in} $0$ \\[.1in]
4 & \hspace{-.1in} $t$ &\hspace{0.1in} $0$ &\hspace{0.1in}  $0$ &\hspace{0.1in} $0$
\end{tabular}
\end{minipage}
}
\hspace{-0.8in} 
\resizebox{.75\textwidth}{!}{\begin{minipage}{\textwidth}
\begin{tabular}{l|ccccc}
$q^4_{ij}$ & \hspace{.01in}  1 &\hspace{.01in} 2 &\hspace{.01in} 3 &\hspace{.01in} 4\\[.1in]
\hline\\[.03in] 
1 & \hspace{-.1in} $0$  &\hspace{.01in} $0$ &\hspace{.01in}  $t s-1$ &\hspace{.01in}  $0$ \\[.1in]
2 & \hspace{-.1in} $0$ &\hspace{.01in}  $(s+t)s$ &\hspace{.01in} $0$ &\hspace{.01in}  $0$   \\[.1in]
3 & \hspace{-.1in} $t s-1$  &\hspace{.01in} $0$ &\hspace{.01in}  $0$ &\hspace{.01in}  $0$ \\[.1in]
4 & \hspace{-.1in} $0$ &\hspace{.01in}  $0$ &\hspace{.01in} $0$ &\hspace{.01in}  $0$  
\end{tabular}
\end{minipage}
}

\vspace{0.5in}

\noindent
Our idea is to reconstruct the quadrangle $\cS$  with the hemisystem $U$ starting from the parameters of the scheme $\cX=(\cL,\{R_i\}_{i=0}^{4})$.

\section{Some properties of maximal $\{0,1,2\}$-cliques of $\cX$}\label{sec_3}

Coolsaet and Juri\v{s}i\'c \cite{cj} introduced some parameters that in a certain sense generalize the intersection numbers $p_{ij}^k$, so requiring \emph{greater regularity} of the scheme. We recall here some properties and observations from \cite{cj, ms}.

Let $xyu$ be a (ordered) triple of elements in  $X$, and  $l,m,n \in \{0,\ldots,d\}$. The {\em triple intersection number} $\left[ \begin{smallmatrix} x & y & u \\l & m & n \end{smallmatrix} \right]$ (or $[\, l\; m\; n \,]$ for short) denotes the number of vertices $z\in X$ such that 
\[
(x,z)\in R_l, \ \ (y,z)\in R_m, \ \ (u,z)\in R_n.
\] 

\noindent
Note that this symbol is invariant under permutations of its columns. Let $
(x,y)\in R_A$,  $(y,u)\in R_B$, $(u,x)\in R_C$, for some $A,B,C\in\{0,\ldots,d\}$.
The following identities hold:
\begin{equation}\label{eq_3}
\begin{array}{ccccccccc}
[\, 0\; m\; n \,] &+ &[\, 1\; m\; n \,]&+&\ldots &+&[\, d\; m\; n \,]& =&p^B_{mn}, \\[.02in]
[\, l\; 0\; n \,] &+ &[\, l\; 1\; n \,]&+&\ldots&+&[\, l\; d\; n \,]& =&p^C_{ln}, \\[.02in]
[\, l\; m\; 0 \,] &+& [\, l\; m\; 1 \,]&+&\ldots&+&[\, l\; m\; d \,]& =&p^A_{lm},
\end{array}
\end{equation}
for all $l,m,n\in\{0,\ldots,d\}$. It is worth observing that, when in one of the above equations the right-hand side is zero, then every triple intersection number on the left-hand side is zero too. This allows us to reduce considerably the forthcoming calculations (in particular in the next lemma).  Further, since
\begin{equation}\label{eq_2}
 [\, 0\; m\; n \,] = \delta_{m\,A}\delta_{n\, C},\ \ \  
 [\, l\; 0\; n \,] = \delta_{l\,A}\delta_{n\, B},\ \ \ 
 [\, l\; m\; 0 \,] = \delta_{l\,C}\delta_{m\, B}, 
\end{equation}
identities (\ref{eq_3}) reduce to
\begin{equation}\label{eq_4}
\begin{aligned}
\sum_{r=1}^{d}{[\, r\; m\; n \,]}& =p^B_{mn}-\delta_{mA}\delta_{n C}, \\[.02in]
\sum_{r=1}^{d}{[\, l\; r\; n \,]}& =p^C_{ln}-\delta_{lA}\delta_{n B},\\[.02in]
\sum_{r=1}^{d}{[\, l\; m\; r \,]}& =p^A_{lm}-\delta_{lC}\delta_{m B},
\end{aligned}
\end{equation}
for all $l,m,n\in\{1,\ldots,d\}$. 
Identities  (\ref{eq_4}) can be interpreted as a system of  equations in the $d^3$ non-negative unknowns $[\, l\; m\; n\,]$, and we refer to it as {\em the system of equations associated with the triplet} $ABC$.  We direct the reader to Coolsaet and Juri\v{s}i\'c \cite{cj} for more details on triple intersection numbers in an association scheme.

Note that the system is uniquely determined by its  constant terms array, which will be denoted by $\left(p^B_{mn}-\delta_{mA}\delta_{n C};p^C_{ln}-\delta_{lA}\delta_{n B};p^A_{ln}-\delta_{lC}\delta_{m B}\right)_{l,m,n\in\{0,\ldots,d\}}$. Instead of the (ordered) triple $xyu$, consider the triple $yxu$. Then, the system of linear equations associated with $A'B'C'=ACB$ is defined by the following constant terms array
\begin{equation}\label{eq_5}
\left(p^C_{mn}-\delta_{mA}\delta_{n B};p^B_{ln}-\delta_{lA}\delta_{n C};p^A_{ln}-\delta_{lB}\delta_{m C}\right).
\end{equation}  
The unknown $[\, i\; j\; k \,]$ in (\ref{eq_4}),  that represents  $\left[ \begin{smallmatrix} x & y & u \\i & j & k \end{smallmatrix} \right]$, corresponds to the unknown $[\, j\; i\; k \,]'$ in the system defined by (\ref{eq_5}), representing  $\left[ \begin{smallmatrix} y & x & u \\j & i & k \end{smallmatrix} \right]$.  Note that if $B=C$, the two systems coincide, so that $[\, i\; j\; k \,]=[\, j\; i\; k \,]'=[\, j\; i\; k \,]$, for all $i,j\in\{0,\ldots,d\}$. This yields more useful conditions on unknowns in (\ref{eq_4}).
\begin{lemma}\label{lem_2}
For the scheme $\cX$, the following identities are satisfied:
\begin{enumerate}
\item[i.] if  $(x,y),(y,u),(u,x)\in R_2$, then $[\, 2\; 2\; i\,] =0$  for  $i=1,3,4$ and $[\, 2\; 2\; 2\,] =\frac12(t-5)$ for $t\ge 5$;
\item[ii.] if $(x,y)\in R_2$ and $(y,u),(u,x)\in R_1$, then $[\, 1\; 1\; i\,] =0=[\, 2\; i\; 1\,]$ for  $i=1,3,4$,  $[\, 1\; 1\; 2\,] =\frac12(t-1)$ and $[\, 2\; 2\; 1\,] =\frac12(t-3)$.
\end{enumerate}
\end{lemma}
 \begin{proof}
 As $\cX$ is a $Q-$polynomial $Q-$antipodal scheme with $q^k_{ij}=0$, for $( i\, j\, k)= (1\,1 \, 3) , (1\,1 \, 4),  (1\,4 \, 2), (1\,4\, 4)$   and their permutations, by \cite[Theorem 3]{cj}, we also have
\begin{equation}\label{eq_6}
\sum_{l,m,n=1}^{4}{Q_{lr}Q_{ms}Q_{nt}[\, l\; m\; n\,]}=-Q_{0r}Q_{As}Q_{Ct}-Q_{Ar}Q_{0s}Q_{Bt}-Q_{Cr}Q_{Bs}Q_{0t},
\end{equation}
for $(r\, s\, t)=  (1\,1 \, 3) , (1\,1 \, 4),  (1\,4 \, 2), (1\,4\, 4)$   and their permutations. Also identities (\ref{eq_6}) can be interpreted as a system of equations in the $4^3=64$ non-negative unknowns $[\, l\; m\; n\,]$. 

For $A=B=C=2$,  we widen the system (\ref{eq_4}) with  the identities $[\,l\;m\;n\,]=[\,\sigma(l)\;\sigma(m)\;\sigma(n)\,]$, for any permutation $\sigma$ on symbols $l,m,n$, and  Equations (\ref{eq_6}). 
Handing the above equations  to the computer algebra system Mathematica \cite{math}, we obtain their space of solutions,  depending  on $[\, 1\; 1\; 2\,]$, $[\, 1\; 1\; 3\,]$, $[\, 1\; 1\; 4\,]$, $[\, 1\; 2\; 2\,]$,  $[\, 1\; 2\; 3\,]$, $[\, 1\; 3\; 4\,]$, $[\, 2\; 3\; 4\,]$,  which are all zero for $t\ge 5$. In particular,  
\[
\begin{array}{rcl}
[\, 2\; 2\; 1\,] & = & [\, 1\; 2\; 2\,]=0,\\[.06in]
[\, 2\; 2\; 3\,] & = &
0,\\[.09in]
[\, 2\; 2\; 4\,] & = & -\frac{(5t+1)(t-1)}{(t+1)^2}[\, 1\; 1\; 2\,]-[\, 1\; 1\; 3\,]-2\frac{(t-1)^2}{(t+1)^2}([\, 1\; 1\; 4\,]+[\, 1\; 2\; 3\,])\\[.09in]
&&-\frac{t-1}{t+1}[\, 1\; 2\; 2\,] +2\frac{t-3}{(t+1)^2}[\, 1\; 3\; 4\,]-2\frac{1}{t+1}[\, 2\; 3\; 4\,]\\[.09in]&=&0,\\[.09in]
[\, 2\; 2\; 2\,] & = & \frac12(t-5)+2\frac{t^3-7t^2+t+1}{(t-1)(t+1)^2}[\, 1\; 1\; 2\,]+[\, 1\; 1\; 3\,]+\frac{t^2-6t+1}{(t+1)^2}[\, 1\; 1\; 4\,]+
\frac{t^2+3}{t^2-1}[\, 1\; 2\; 2\,]\\[.09in]
&&-\frac{8t}{t^2+1}[\, 1\; 2\; 3\,]+4\frac{3t-1}{(t-1)(t+1)^2}[\, 1\; 3\; 4\,]-4\frac{1}{(t-1)(t+1)}[\, 2\; 3\; 4\,]\\[.09in]&=&\frac12(t-5).
\end{array}
\]
2. For $A=2$, $B=C=1$,  we widen the system (\ref{eq_4}) with  the identities $[\,l\;m\;n\,]=[\,m\;l\;n\,]$, and  Equations (\ref{eq_6}). 
As before,  thanks to  Mathematica \cite{math}, we obtain their space of solutions depending on $[\, 1\; 2\; 1\,]$, $[\, 1\; 2\; 3\,]$, $[\, 1\; 2\; 4\,]$, $[\, 1\; 3\; 2\,]$,   $[\, 1\; 3\; 4\,]$, $[\, 1\; 4\; 1\,]$, $[\, 1\; 4\; 2\,]$, $[\, 1\; 4\; 3\,]$, $[\, 2\; 3\; 1\,]$, $[\, 2\; 3\; 2\,]$, $[\, 2\; 3\; 4\,]$, $[\, 2\; 4\; 4\,]$, $[\, 3\; 4\; 1\,]$, $[\, 3\; 4\; 2\,]$, $[\, 3\; 4\; 4\,]$, $[\, 4\; 4\; 2\,]$, which are all zero except for $[\, 1\; 3\; 4\,]=\frac12t^2(t+1)$. This implies  
\[
\begin{array}{rcl}
[\, 1\; 1\; 1\,]
&=& 0,\\[.09in]
[\, 1\; 1\; 3\,] 
&=& 0,\\[.09in]
[\, 1\; 1\; 4\,]& = & \frac12t^2(t+1)-[\, 1\; 3\; 4\,]\\[.09in]
&=& 0,\\[.09in]
[\, 1\; 1\; 2\,] & = & \frac12{(t-1)}.
\end{array}
\]
Furthermore, we have  
\[
\begin{array}{rcl}
[\, 2\; 1\; 1\,] &= & [\, 1\; 2\; 1\,]= 0,\\[.09in]
[\, 2\; 4\; 1\,]& = & \frac{t^2(t+1)(t-1)^2}{2(t^3+1)}+\frac{(t-1)^2}{t^3+1}[\, 1\; 3\; 4\,]\\[.09in]
&=& 0,\\[.09in]
[\, 2\; 2\; 1\,] & = & -\frac{t^2(t+1)(t-1)^2}{2(t^3+1)}-\frac{(t-1)^2}{t^3+1}[\, 1\; 3\; 4\,]\\[.09in]
&=& \frac12{(t-3)}.
\end{array}
\]
\end{proof}

Let $\cX=(X,\{R_i\}_{i=0}^{d})$ be a (symmetric) association scheme. 
\noindent

For $F \subseteq \{ 0,1,\ldots,d\}$, an \emph{$F$-clique in $\mathcal{X}$} is any $C \subseteq X$ with $C \times C \subseteq \cup_{f \in F} R_f$. We now apply the above results to particular $F$-cliques in $\mathcal{X}$.

For any $x,y \in X$ and $i,j=0,\ldots, d$, we set $\cP^{(x,y)}_{i,j}=\{z\in X:(x,z)\in R_i, (y,z)\in R_j\}$.

\begin{proposition}\label{lem_3}
Let $x,y\in X$ with $(x,y)\in R_2$. Then, there exists a unique maximal $\{0,2\}$-clique (of size $(t+1)/2$) containing $x$ and $y$, and a unique $\{0,2\}$-clique (of size $(t+1)/2$) whose vertices are all $1-$related to both $x$ and $y$. These cliques are the sets $C=\{x,y\}\cup\cP^{(x,y)}_{2,2}$ and $C'=\cP^{(x,y)}_{1,1}$, where $\cP^{(x,y)}_{i,i}=\{z\in X:(x,z),(y,z)\in R_i\}$, respectively.
\end{proposition}
\begin{proof}
 Clearly, $|\cP^{(x,y)}_{i,i}|=p^{2}_{i,i}=\frac{t-3}{2} +2-i$, with $i=1,2$. Let $t\ge 5$. Assume there exists a pair of distinct elements  $z,z'\in \cP^{(x,y)}_{2,2}$, with $(z,z')\in R_k$, $k\in\{1,3,4\}$.  Since $p^k_{22}=0$, for $k=1,3$, it follows that $k=4$. Then, $xyzz'$ is a $4-$tuple such that $x,y,z$ are mutually $2-$related and 
\[
(z',x),(z',y)\in R_2, \ \ (z',z)\in R_4.
\] 
This yields that the number $[\, 2\; 2\; 4\,]$ should be nonzero. But this contradicts  Lemma \ref{lem_2}.i. Then, $\{x,y\}\cup\cP^{(x,y)}_{2,2}$ is the unique maximal $\{0, 2\}$-clique of size $2+p^2_{22}=(t+1)/2$ containing $x$ and $y$. If $t=3$, as $p^2_{22}=0$, the latter assertion is trivially true. 

Next assume there exists a pair of distinct elements  $z,z'\in \cP^{(x,y)}_{1,1}$, with $(z,z')\in R_k$, $k\in\{1,3,4\}$. Since $p^k_{11}=0$, for $k=1,3$,  it derives that $k=4$. Then, $xyzz'$ is a $4-$tuple such that $(x,y)\in R_2$,$(x,z),(y,z)\in R_1$ and 
\[
(z',x),(z',y)\in R_1, \ \ (z',z)\in R_4.
\] 
This yields that the number $[\, 1\; 1\; 4\,]$ should be nonzero. But this contradicts  Lemma \ref{lem_2}.ii. Then, $\cP^{(x,y)}_{1,1}$ is the unique maximal $\{0,2\}$-clique of size $p^2_{11}=(t+1)/2$ with the desired property.
\end{proof}
\begin{remark}\label{rem_1}{\em 
From Lemma \ref{lem_2}.ii., the set $C\cup C'$ can be constructed also by starting from a pair  $(x,u)\in R_1$, with $x\in C$ and $u\in C'$. In fact,  $C=\{x\}\cup\cP^{(x,u)}_{2,1}$ and $C'=\{u\}\cup\cP^{(x,u)}_{1,2}$.}
\end{remark}
It follows that $C\cup C'$ is a maximal $\{0,1,2\}-$clique of size $t+1$, and from now  ``clique'' will stand for  such a clique.

\section{Reconstructing the generalized quadrangle from the scheme}\label{sec_4}

We now formulate and prove the main result.
 
\begin{theorem} \label{main}
Let $\cC$ be the set of all cliques $C\cup C'$ constructed in the previous section. 
 Then, the incidence structure $\cS=(\cC,X,{\rm I})$, where the incidence relation ${\rm I}$ is the containment, is a GQ of order $(t^2,t)$. Moreover, the set $U=\{x\}\cup R_2(x)\cup R_4(x)$, where $x\in X$ is arbitrarily chosen, is a hemisystem of $\cS$. 	
\end{theorem}
\begin{proof}
Since there is a point-line duality for a GQ, for which in any definition or theorem the words ``point'' and ``line'' are interchanged, we will prove the theorem in the   dual GQ $\cS^D=(X,\cC,{\rm I})$ of order $(t,t^2)$. By Lemma \ref{lem_2},  each clique has size $t+1$, and for each $(x,y)\in R_1\cup R_2$ there is a unique clique. Since $n_1+n_2=t(t^2+1)$, by Proposition \ref{lem_3} and Remark \ref{rem_1}, it is evident that  there are $t^2+1$ cliques through each $x\in X$, and two cliques intersect in at most one point.

Let $x\in X$ and $C\cup C'\in \cC$, with $x\notin C\cup C'$. Suppose there is $z\in C\cup C'$ such that $(x,z)\in R_4$, and consider $\cP^{(z,x)}_{i,j}=\{u\in X:(u,z)\in R_i, (u,z)\in R_j\}$ with $|\cP^{(z,x)}_{i,j}|=p^4_{ij}$, $i,j\in\{1,2\}$. Note that $p^4_{12}=p^4_{2,1}=0$, $p^4_{22}+p^4_{1,1}=\frac12 (t-1)^2+\frac12(t+1)^2=t^2+1$. We show that if $(z,x)\in R_4$, then $x$ is $2-$related to a unique point on $p^4_{22}$ cliques through $z$, and $1-$related to a unique point on the remaining cliques through $z$. The first part of the assertion follows from the fact that if $x$ is  $2-$related to more then one element of the same clique through $z$, then it would be on that clique by Proposition \ref{lem_3}. It remains to check that neither $(x,z_1),(x,z_2)\in R_1$ nor $(x,z_1)\in R_1$, $(x,z_2)\in R_2$ with $z_1\neq z_2$ on the same clique through $z$ can exist. Suppose it is possible. In the first case, as $p^1_{11}=0$, we would have $(z_1,z_2)\in R_2$ which would imply $x$ belonging to the clique on $z$ with $z_1$ and $z_2$, by Proposition \ref{lem_3}. In the second case, as  $p^2_{21}=0$, $z_1$ and $z_2$ would be $1-$related to each other, and this, again by Proposition \ref{lem_3}, would led $x$ to belong to the clique containing $z_1$ and $z_2$. Finally, if $x\notin C\cup C'$ is $4-$related to a $z\in C\cup C'$, there will be a unique element on $C\cup C'$ $1-$ or $2-$related to $x$. Next, suppose there is $z\in C\cup C'$ $3-$related to $x$. As $p^3_{11}=p^3_{2,2}=0$, $p^3_{12}+p^3_{21}=t^2+1$, similar arguments lead to the same conclusion.

Let $U_x=\{x\}\cup R_2(x)\cup R_4(x)$. We will see that $U_x=U_y$, for any $y\in U_x$. Let $y\in R_2(x)$ and $z\in R_2(y)$. As $p^2_{12}=p^2_{32}=0$, the only possibility is that $(x,z)\in R_2 \cup R_4$. This implies that $z\in R_2(x)\cup R_4(x)$. We get the same conclusion for $z\in R_4(y)$ as $p^2_{14}=p^2_{34}=0$. Since $p^4_{12}=p^4_{32}=p^4_{14}=p^4_{34}=0$, similar arguments can be used when $y\in R_4(x)$. Thus, $U_y=U_x$ as $|U_x|=|U_y|$, from which we may set $U=U_x$.

Consider $z\in U$ and set $U=U_z$. Let $C\cup C'$ any given clique and $(x,y)\in R_1$, $x\in C$, $y\in C'$. As $p^1_{11}=p^1_{1,3}=p^1_{33}=0$, at least one between $x$ and $y$ must be $2-$ or $4-$related  to $z$, say $x$. Therefore, $x\in U_z$, that is $U_x=U$. So, $C$ is contained in $U$, and $C'\cap U=\emptyset$. In conclusion $U$ is a set of points of $\cS^D$ such that very line meets it in $(t+1)/2$ points. Under duality, $\cS$ is a GQ of order $(t^2,t)$ with $U$ a hemisystem in $\cS$.  
\end{proof}

\end{document}